\documentclass[12pt]{article}

\usepackage[english]{babel}
\usepackage[utf8]{inputenc}
\usepackage[T1]{fontenc}
\usepackage{lmodern}
\usepackage{microtype}
\usepackage{indentfirst}

\usepackage[hmargin=3cm,vmargin=3.5cm]{geometry}

\usepackage{amsmath, amssymb, amsthm}
\usepackage{eucal}
\usepackage{stmaryrd}
\usepackage{braket}
\usepackage[all]{xy}
\usepackage{prooftree}
\usepackage{mathtools}

\usepackage{enumitem}
\usepackage{booktabs}

\usepackage[capitalize]{cleveref}
\usepackage{verbatim}
\usepackage{xspace}
\usepackage{emptypage}
\usepackage{url}

\theoremstyle{definition}
\newtheorem{defin}{Definition}[section]
\newtheorem{rem}[defin]{Remark}

\theoremstyle{plain}
\newtheorem{theor}[defin]{Theorem}
\newtheorem{prop}[defin]{Proposition}
\newtheorem{lemma}[defin]{Lemma}
\newtheorem{corol}[defin]{Corollary}

\newcommand{\mltt}{Martin-L\"of type theory\xspace}
\newcommand{\type}{\ \mathsf{Type}}
\newcommand{\cxt}{\ \mathsf{Cxt}}
\newcommand{\T}{$\mathbb{T}$\xspace}
\newcommand{\Tid}{$\mathbb{T}_0$\xspace}
\newcommand{\Id}[2]{\textup{Id}_{#1}(#2)}
\newcommand{\refl}{\textup{refl}}
\newcommand{\J}{\textup{J}}

\newcommand{\subst}[2]{[^{#1}\!/\!_{#2}]}
\newcommand{\bnd}[1]{{\scriptstyle[#1]}}

\newcommand{\catt}{$\mathsf{C}$\xspace}	\newcommand{\catm}{\mathsf{C}}

\newcommand{\Lt}{$\mathcal{L}$\xspace}	\newcommand{\Lm}{\mathcal{L}}
\newcommand{\Rt}{$\mathcal{R}$\xspace}	\newcommand{\Rm}{\mathcal{R}}
\newcommand{\At}{$\mathcal{A}$\xspace}	\newcommand{\Am}{\mathcal{A}}
\newcommand{\bsch}[2]{#1^*\!#2}
\newcommand{\pr}{\textup{pr}}
\newcommand{\univ}[1]{\langle #1 \rangle}
\newcommand{\pob}{\textup{Path}}
\newcommand{\rar}{\textup{r}}
\newcommand{\mpob}{\textup{Map}}
\newcommand{\PM}{\textup{pm}}
\newcommand{\id}{\textup{id}}
\newcommand{\term}{\mathbf{1}}

\newcommand{\Mt}{$\textup{M}$\xspace}	\newcommand{\Mm}{\textup{M}}
\newcommand{\Idf}{\textup{Id}}
\newcommand{\src}{\textup{s}}
\newcommand{\trg}{\textup{t}}
\newcommand{\cmp}{\textup{m}}
\newcommand{\natto}{\Rightarrow}

\newcommand{\syntc}{$\mathsf{C}(\mathbb{T})$\xspace}
\newcommand{\syntcid}{$\mathsf{C}(\mathbb{T}_0)$\xspace}
\newcommand{\Dm}{\mathcal{D}}
\newcommand{\Dt}{$\mathcal{D}$\xspace}

\newcommand{\preal}{\mathbb{R}_+}

\newcommand{\presc}[2]{\ {}_{#1}\!\!#2}
\SelectTips{cm}{}
\newdir{ >}{{}*!/-7pt/@{>}}
\newcommand{\xycenter}[2]{\vcenter{\hbox{\xymatrix@#1{#2}}}}

\renewcommand{\phi}{\varphi}

\title{A category-theoretic version of the identity type weak factorization system}
\author{Jacopo Emmenegger}
\date{March 24, 2015}

\begin{document}

\maketitle

\section{Introduction}										\label{sec:intro}

In recent years much interest has been devoted to investigating the relation between
paths in a topological space, and proofs of identity in \mltt.
In particular, Steve Awodey and Michael Warren have suggested in~\cite{AW09}
that a suitable structure to interpret the identity type rules is a weak factorization system,
a structure arising in homotopy theory as an abstraction from the notion of path lifting property
and which is part of the stronger notion of model category~\cite{Qu67,DS95}.
Some coherence issues have then been worked out by Benno van den Berg and Richard Garner in~\cite{vdBG11},
where they introduced the notion of path object category
and proved that it provides a weak factorization system with good enough features
to obtain a sound interpretation of identity types.
Furthermore, Nicola Gambino and Richard Garner have proved in~\cite{GG08} a somewhat inverse result,
stating that the syntactic category (or category of contexts) of a dependent type theory with identity types
carries a weak factorization system structure.

This highlights a deep connection between identity types and weak factorization systems
but, as already pointed out in~\cite{vdBG11}, a gap still remains,
namely that the syntactic category (and in general the models build out from the syntax of \mltt)
does not fit into the kind of models described by Awodey, van den Berg, Garner and Warren,
thus keeping out any possibility of a completeness result for such models.

The goal of this paper is to abstract from the particular setting of the syntactic category
in order to generalize the result of Gambino and Garner
and investigate how it relates with the models presented in~\cite{vdBG11}.
The obtained structure turns out to subsume both
the syntactic category and the homotopy-theoretic models of~\cite{vdBG11}.
In particular, this paper can be regarded as a development of the thesis~\cite{Em13}
written under the supervision of Nicola Gambino.

Our starting point is a particular category, introduced by André Joyal in~\cite{Jo13}, called tribe.
This is a category with a distinguished class of arrows \At
which, roughly speaking, can be thought of as abstract dependent projections.
We shall then consider additional structure on it, defining a tribe with weakly stable path objects,
and prove that such category admits a weak factorization system, thus giving rise to Joyal's notion of h-tribe.

Independently of this work, Benno van den Berg has announced similar results in a recent talk in Oxford~\cite{vdB14},
achieved by considering a structure, called identity tribe, closely related to a tribe with path objects.
The main difference between our two notions lies in the condition of stability under pullbacks
imposed on the factorization of the diagonal arrow.
Nevertheless, their similarities suggest that we can regard them
as leading to a natural enrichment of the tribe structure on a category.

We begin \cref{sec:wspob} recalling some basic definitions and facts about weak factorization systems.
The notion of tribe with weakly stable path objects is then defined,
and its relation with the notions introduced by Joyal and van den Berg is discussed.
The existence of a weak factorization system $(\Lm, \Rm)$ is proved in~\cref{thm:wfs},
and we conclude the section characterizing \Rt
as the class of arrows having a form of transport between fibres over the endpoints of a path.
The existence of a weak factorisation system from an analogous setup has been recently proved by Michael Shulman~\cite{Sh12}.

In \cref{sec:syntc} we briefly recall the definition of the syntactic category \syntc and
prove that, if the underlying type theory \T has identity types,
then \syntc is a tribe with weakly stable path objects.
We then show that, in this case, \cref{thm:wfs} yields exactly
the identity type weak factorization system of~\cite{GG08}.

We conclude in \cref{sec:pobc} recalling the definition of path object category
along with the underlying topological intuition,
and show how to adapt the construction of a weak factorization system in~\cite{vdBG11}
to prove that every path object category is a tribe with weakly stable path objects.

\section{Weak factorization systems and weakly stable path objects}				\label{sec:wspob}

In this section we begin recalling the definition of weak factorization system and a couple of elementary results.
We present then the definition of tribe and introduce the richer structure of tribe with weakly stable path objects,
also discussing its relation with Joyal's h-tribe and van den Berg's identity tribe.
Finally, we prove our main result,
stating that every tribe with weakly stable path objects can be endowed with a weak factorization system.

\begin{defin}
Let \catt be a category.
Given two arrows $f \colon A \to X$ and $g \colon B \to Y$, a \emph{left lifting problem for $f$ over $g$}
(or a \emph{right lifting problem for $g$ over $f$}) is a commutative square
\[\xycenter{=2em}{
	A 	\ar[d]_f 	\ar[r] 	&	 B 	\ar[d]^g 	\\
	X 		\ar[r] 		&	 Y		 	}\]
The arrow $f$ has the \emph{left lifting property (l.l.p.)} with respect to $g$,
if every left lifting problem for $f$ over $g$ has a \emph{diagonal filler},
that is an arrow $j \colon X \to B$ such that the diagram
\[\xycenter{=2em}{
	A 	\ar[d]_f	\ar[r] 	&	 B 	\ar[d]^g	\\
	X 	\ar[r] 	\ar[ur]^j	&	 Y			}\]
commutes.
Symmetrically, $g$ has the \emph{right lifting property (r.l.p.)} with respect to $f$,
if every right lifting problem for $g$ over $f$ has a diagonal filler.
Of course, $f$ has the l.l.p.\ with respect to $g$ if and only if $g$ has the r.l.p.\ with respect to $f$,
we write $f \boxslash g$ to denote this situation.

If \At is a class of arrows, we write $\Am^\boxslash$ for the class of all the arrows that have the r.l.p.\
with respect to every arrow in \At, and $^\boxslash\Am$ for the class of all the arrows that have the l.l.p.\
with respect to every arrow in \At.
We shall also write $f \boxslash \Am$ instead of $f \in {}^\boxslash\!\Am$.
\end{defin}

\begin{defin}
A \emph{weak factorization system (w.f.s.)} on \catt consists of a pair of classes of arrows $(\Lm, \Rm)$ such that:
\begin{enumerate}[label=(\emph{\alph*})]
	\item every arrow $f$ in \catt admits a factorization $f = pi$, where $i \in \Lm$ and $p \in \Rm$,
	\item $\Lm^\boxslash = \Rm$ and $^\boxslash\Rm = \Lm$.
\end{enumerate}
These properties are called Factorization Axiom and Weak Orthogonality Axiom, respectively.
\end{defin}

\begin{lemma} \label{lem:wort}
Let \catt be a category and $\Am \subset \textup{Ar}(\catm)$ a class of arrows in \catt.
Define $\Lm \coloneqq {}^\boxslash\!\Am$ and $\Rm \coloneqq \Lm^\boxslash$.
Thus $\Lm = {}^\boxslash\Rm$.
\end{lemma}
\begin{proof}
Let $f \in \Lm$. Since every arrow in \Rt has the r.l.p.\ with respect to all the arrows in \Lt,
in particular every $g \in \Rm$ has the r.l.p. with respect to $f$.
Hence, from the symmetry of lifting properties, $f$ has the l.l.p with respect to any $g \in \Rm$.
Thus, $f \in {}^\boxslash\Rm$.

If we apply the same argument to an arrow in \At, we obtain $\Am \subset \Rm$.
Thus we have the opposite inclusion ${}^\boxslash\Rm \subset {}^\boxslash\!\Am = \Lm$.
\end{proof}

\begin{lemma} \label{lem:plbklp}
Let \catt be a category and $\Am \subset \textup{Ar}(\catm)$ a class of arrows in \catt.
If \catt has pullbacks and \At is closed under base change, then $f \boxslash \Am$
if and only if a diagonal filler exists for every lifting problem of the form
\[\xycenter{=2em}{
	A 	\ar[d]_f 	\ar[r] 	&	 B 	\ar[d]^{\in\Am} 	\\
	X 		\ar@{=}[r]	&	 X		 	}\]
\end{lemma}
\begin{proof}
If $f \boxslash \Am$, then obviously a diagonal filler exists.
To prove the opposite implication notice that, given a left lifting problem for $f$ over an arrow in \At,
taking the pullback of the bottom and the right-hand arrow yields a left lifting problem with the identity arrow
as the bottom arrow, as illustrated by the following diagram
\[\xycenter{C=3em}{
	A 	\ar[d]_f \ar[r]^{\univ{f,\,h}} 	&	P	\ar[d]\ar[r]	&	 B 	\ar[d]^g 	\\
	X	 	\ar@{=}[r]		&	X	\ar[r]_k 	&	 Y		 	}\]
where the right-hand square is a pullback, $\univ{f,h}$ is the arrow given by its universal property
and the outer square is the original lifting problem.
Since $g \in \Am$ and \At is closed under base change, the left-hand square has a diagonal filler.
Composing it with the base change of $k$ along $g$ yields, in turn,
a diagonal filler for the original lifting problem.
\end{proof}

\begin{defin}											\label{def:tribe}
Let \catt be a category and $\Am \subset \textup{Ar}(\catm)$ a class of arrows in \catt.
The pair $(\catm, \Am)$ is a \emph{tribe} if \catt has a terminal object $\term$ and the following hold:
\begin{enumerate}[label=(\emph{\alph*})]
	\item for every pair of arrows in \catt with the same codomain,
		we have a choice of a pullback square if at least one of them is in \At,	\label{def:tribe:a}
	\item \At is closed under composition and base change,					\label{def:tribe:b}
	\item all the iso and terminal arrows are in \At.					\label{def:tribe:c}
\end{enumerate}
\end{defin}

Given two arrows $p \in \Am$ and $f$ with the same codomain, we write as usual $\bsch{f}{p}$
to denote the (chosen) base change arrow of $p$ along $f$ and similarly for $\bsch{p}{f}$.
Notice that point~\ref{def:tribe:b} of \cref{def:tribe} implies $\bsch{f}{p} \in \Am$.
The arrow defined by the universal property of a pullback is denoted by $\univ{f,g}$.
Furthermore, we call \emph{products} those pullback squares involving a single arrow $p \in \Am$ or two terminal
arrows, \emph{projections} the base change arrows and denote them by
\[\xycenter{C=3.5em}{
	E \times_p E 	\ar[d]^{\pr^p_0}	\ar[r]_(.6){\pr^p_1}	&	E 	\ar[d]^p 	&&
			X \times Y 	\ar[d]^{\pr_0}		\ar[r]_(.6){\pr_1}	&	Y	\ar[d]	\\
	E 			\ar[r]_p 				&	Y			&&
			X			\ar[r]					&	\term		}\]
respectively.
We shall also denote by $f \times g \coloneqq \univ{f \pr_0, g \pr_1}$ the product arrow between products.
Notice that points~\ref{def:tribe:b} and~\ref{def:tribe:c} of \cref{def:tribe}
imply that both the projections are in \At.
We shall drop superscripts from the projections whenever they are clear from the context.

\begin{defin}											\label{def:pobtr}
We say that a tribe $(\catm, \Am)$ has \emph{path objects} if, for every arrow $p \colon E \to Y$ in \At,
we have a choice of a factorization of the diagonal
$\Delta_p \coloneqq \univ{\id_E, \id_E} \colon E \to E \times_p E$,
denoted by
\[\xycenter{C=3.5em}{
	E \ar[r]^(.45){\rar_p}	&	 \pob(p)	\ar[r]^{\partial_p} &	 E \times_p E, 	}\]
such that
\begin{enumerate}[label=(\emph{\alph*})]
	\item $\partial_p \in \Am$,								\label{def:pobtr:a}
	\item every base change of $\rar_p$ along an arrow in \At is in $^{\boxslash}\!\Am$.	\label{def:pobtr:b}
\end{enumerate}
If $p \colon E \to \term$ is a terminal arrow,
we write $\pob(E)$, $\rar_E$ and $\partial_E$,
instead of $\pob(p)$, $\rar_p$ and $\partial_p$, respectively.
\end{defin}

\begin{defin}											\label{def:wspobtr}
We say that a tribe $(\catm, \Am)$ with path objects has \emph{weakly stable path objects} if,
for every arrow $p \colon E \to Y$ in \At and every arrow $f \colon X \to Y$, there exists an arrow
\[	i \colon \bsch{f}{\,\pob(p)} \to \pob(\bsch{f}{p})	\]
such that the diagram
\begin{equation}										\label{wspobtr:sq}
\xycenter{=4em@C=8em}{
	\bsch{f}{E}	\ar@{=}[d]	\ar[r]^(.46){\univ{\bsch{f}{p}, \, \rar_p.\bsch{p}{f}}}	&
		\bsch{f}{\,\pob(p)}	\ar[d]^i
			\ar[r]^(.46){\univ{\bsch{f}{(p.\partial_p^0)}, \, \partial_p.\bsch{(p.\partial_p^0)}{f}}}
			&	\bsch{f}{E} \times_{\bsch{f}{p}} \bsch{f}{E}		\ar@{=}[d]		\\
	\bsch{f}{E}	\ar[r]^(.46){\rar_{\bsch{f}{p}}}	&
		\pob(\bsch{f}{p})    \ar[r]^(.46){\partial_{\bsch{f}{p}}}	&
			\bsch{f}{E} \times_{\bsch{f}{p}} \bsch{f}{E}	\\}
\end{equation}
commutes.
\end{defin}

Joyal defines a h-tribe as a tribe $(\catm, \Am)$ in which every arrow can be factored through an arrow in
${}^\boxslash\!\Am$ followed by an arrow in \At, and such that the class ${}^\boxslash\!\Am$
is closed under base change.
Since it is possible to prove that in a tribe with weakly stable path objects the class
${}^\boxslash\!\Am$ is closed under base change,
the proof of \cref{thm:wfs} implies that every tribe with weakly stable path objects is a h-tribe.

The identity tribe proposed by van den Berg's seems to be instead a tighter notion,
as it is defined as a tribe with path objects
in which the factorization of the diagonal arrow is required
to be stable under pullback along any arrow.
It would appear that this is equivalent, in our formulation,
to requiring the arrow $i$ to be an iso.

According to this remark, every identity tribe is a tribe with weakly stable path objects,
and both of them are h-tribes.

We now introduce a couple of useful definitions.

\begin{defin}											\label{def:mpob}
Let $f \colon V \to Y$ be an arrow in \catt, and let $p \colon Y \to X$ be an arrow in \At.
The pullback of $f$ along $\partial^0_p$ will be denoted by
\[\xycenter{=3em@C=4em}{
	\mpob_p(f)	\ar[d]^{\PM^{p,f}_0}	\ar[r]_{\PM^{p,f}_1}	&	\pob(p) \ar[d]^{\partial^0_p}	\\
	V 		\ar[r]_f 					&	Y				}\]
and the object $\mpob_p(f)$ will be called \emph{mapping path object of $f$ along $p$}.
As already said for products, we shall usually drop superscripts from the base change arrows.

We say that \emph{$f$ has a transport arrow over $p$} if the following square
\[\xycenter{=3em@C=4em}{
	V 	\ar[d]_{\univ{\id_V, \rar_p.f}}		\ar@{=}[r]	&	V 	\ar[d]^f 	\\
	\mpob_p(f)		\ar[r]_(.55){\partial^1_p.\PM_1}		&	Y			}\]
has a diagonal filler $j \colon \mpob_p(f) \to V$.

If $p \colon Y \to \term$ is a terminal arrow,
we write $\mpob(f)$ instead of $\mpob_p(f)$,
and simply say ``mapping path object of $f$'' and ``transport arrow for $f$''.
\end{defin}

\begin{rem}											\label{rem:trglift}
Notice that the mapping path object $\mpob(f)$ for an arrow $f \colon V \to Y$ always exists.
If moreover $f$ has a transport arrow $j \colon \mpob(f) \to V$,
given  two arrows $e : W \to V$ and $u \colon W \to \pob(Y)$
such that $f.e = \partial^0_Y.u$, we can define the arrow
$\univ{e, u} \colon W \to  \mpob(f)$ and compose it with the transport arrow $j$,
thus obtaining an arrow $e' \colon W \to V$ such that $f.e' = \partial^1_Y.u$,
i.e.\ an arrow which we can think of as a transport of $e$ along the path $u$.
\end{rem}

\begin{lemma}											\label{lem:transp}
Let $(\catm, \Am)$ be a tribe with path objects and
$E \overset{q}{\longrightarrow} Y \overset{p}{\longrightarrow} X$ be two arrows in \At.
Then the arrow $\univ{id_E, \, \rar_p.q} \colon E \to \mpob_p(q)$ is in ${}^{\boxslash}\!\Am$.
\end{lemma}
\begin{proof}
Let us consider the following commutative diagram
\[\xycenter{=3em@C=5em}{
	E 		\ar[d]^q 		\ar[r]_(.4){\univ{\id_E, \,\rar_p.q}}
		&	\mpob_p(q)	\ar[d]^{\PM_1}	\ar[r]_{\PM_0}	 &	  E 	\ar[d]^q 	\\
	Y \ar[r]_{\rar_p} 	&	 \pob(p) 	\ar[r]_{\partial^0_p} 		& 	Y	}\]
where the right-hand square defines the mapping path object of $q$ along $\partial^0_p$.
Since the composites of the horizontal arrows are identities, the left-hand square is a pullback.
Point~\ref{def:tribe:b} of \cref{def:tribe} thus implies $\PM_1 \in \Am$ and,
in turn, point~\ref{def:pobtr:b} of \cref{def:pobtr} implies
$\univ{\id_E, \rar_p.q} \boxslash \Am$.
\end{proof}

An immediate consequence of this Lemma is that every arrow in \At has a transport arrow.

\begin{corol}											\label{corol:transp}
In the hypothesis of the previous Lemma, the square
\[\xycenter{=3em}{
	E 	\ar[d]_{\univ{id_E, \, \rar_p.q}}	\ar@{=}[r]	&	E 	\ar[d]^q 		\\
	\mpob_p(q)		\ar[r]_(.6){\partial^1_p.\PM_1}		&	Y				}\]
has a diagonal filler $j \colon \mpob_p(q) \to E$.
\end{corol}

The following Lemma gives a characterization of the arrows in the class ${}^\boxslash\!\Am$.

\begin{prop}											\label{prop:lftcl}
Let  $(\catm, \Am)$ be a tribe with path objects and $f \colon V \to Y$ an arrow in \catt.
Then $f \boxslash \Am$ if and only if
there exist two arrows $r \colon Y \to V$ and $\phi \colon Y \to \pob(Y)$ such that
\begin{enumerate}[label=(\roman*)]
\item $r.f = \id_V$,
\item $\partial_Y.\phi = \univ{f.r, \id_Y}$,
\item $\phi.f = \rar_Y.f$.
\end{enumerate}
\end{prop}
\begin{proof}
If $f \boxslash \Am$ then the two arrows are obtained as diagonal fillers of the following
commutative squares:
\[\xycenter{=3em}{
	V 	\ar[d]_f 	\ar@{=}[r]		&	V 	\ar[d]	&&
		V 	\ar[d]_f 	\ar[r]^{\rar_Y.f}		&	\pob(Y)	\ar[d]^{\partial_Y}	\\
	Y	\ar[r]	\ar@{-->}[ur]^r 	&	\term		&&
		Y	\ar[r]_{\univ{f.r, \, \id_Y}}	\ar@{-->}[ur]^{\phi}		&	Y \times Y	}\]
Vice versa, suppose we are given a lifting problem
\[\xycenter{=3em}{
	V 	\ar[d]_f 	\ar[r]^g	&	E 	\ar[d]^q		\\
	Y		\ar@{=}[r]	&	Y 				}\]
where $q \in \Am$.
Since $\partial^0_Y.\phi = f.r = q.g.r$,
we can define the arrow $\univ{g.r, \phi} \colon Y \to \mpob(q)$.
A diagonal filler is then given by $j \univ{g.r, \phi} \colon Y \to E$,
where $j \colon \mpob(q) \to E$ is the transport arrow for $q$ given by \cref{corol:transp}.
The general statement follows from \cref{lem:plbklp}.
\end{proof}

We can use this characterization to prove a generalization of \cref{lem:transp}.

\begin{lemma}											\label{lem:depel}
Let $(\catm, \Am)$ be a tribe with weakly stable path objects and
$V \overset{f}{\longrightarrow} Y \overset{p}{\longrightarrow} X$ be two arrows in \catt.
If $p \in \Am$ then $\univ{\id_V, \rar_p.f} \colon V \to \mpob_p(f)$ is in ${}^\boxslash\!\Am$.
\end{lemma}
\begin{proof}
To simplify notation let us define $M \coloneqq \mpob_p(f)$.
Thanks to \cref{prop:lftcl} and the fact that $\PM_1 \univ{\id_V, \, \rar_p.f} = \id_V$,
it is enough to provide an arrow $\phi \colon M \to \pob(M)$ such that
\begin{equation}										\label{depel:req}
\partial_M.\phi = \univ{\univ{\id_V, \rar_p.f}.\PM_1, \id_M}
\qquad\text{and}\qquad
\phi.\univ{\id_V, \rar_p.f} = \rar_M.\univ{\id_V, \rar_p.f}.
\end{equation}

First of all, let us observe that there is an arrow $v \colon \pob(\PM_0) \to \pob(M)$ obtained
as diagonal filler in
\begin{equation}										\label{depel:v}
\xycenter{=3em@C=5em}{
	M 	\ar[d]_{\rar_{\PM_0}}	\ar[r]^{\rar_M}	&	\pob(M)	\ar[d]^{\partial_M}		\\
	\pob(\PM_0)	\ar[r]_{\univ{\partial^0_{\PM_0},\, \partial^1_{\PM_0}}}
				\ar@{-->}[ur]^v				&	M \times M 		}
\end{equation}
Furthermore, since $\PM_0 = \bsch{f}{\,\partial_p}$, \cref{def:wspobtr} implies the existence of
an arrow
\[	i \colon \bsch{f}{\,\pob(\partial^0_p)} \to \pob(\PM_0).	\]
We can thus exploit the universal property of pullback to define an arrow
$M \to \bsch{f}{\,\pob(\partial^0_p)}$ and compose it with $i$ and $v$ to get the arrow $\phi$.

Let $\psi \colon \pob(p) \to \pob(\partial^0_p)$ be a diagonal filler of the square
\begin{equation}										\label{depel:psi}
\xycenter{=3em@C=5em}{
	Y 	\ar[d]_{\rar_p}		\ar[r]^{\rar_{\partial^0_p}.\rar_p}				&
				\pob(\partial^0_p)		\ar[d]^{\partial_{\partial^0_p}}		\\
	\pob(p)	\ar[r]_(.4){\univ{\rar_p.\partial^0_p, \, \id}}	\ar@{-->}[ur]^\psi	&
				\pob(p) \times_{\partial^0_p} \pob(p)						}
\end{equation}
and consider the following diagram
\begin{equation}										\label{depel:fill}
\xycenter{R=2em@C=5em}{
	M 	\ar[d]_{\PM_0}	\ar[r]^{\PM_1}		&	\pob(p)	\ar@/^1pc/[dr]^\psi	&	\\
	V 	\ar@/_1pc/[dr]^{\univ{\id_V,\, \rar_p.f}}	\ar@/_2pc/[ddr]_{\id_V}			&
				\bsch{f}{\,\pob(\partial^0_p)}	\ar[d]	\ar[r]				&
					\pob(\partial^0_p)		\ar[d]^{\partial^0_{\partial^0_p}}	\\
	&	M 	\ar[d]^{\PM_0}	\ar[r]_{\PM_1}		&	\pob(p)	\ar[d]^{\partial^0_p}	\\
	&	V 	\ar[r]_f 	&	Y 	}
\end{equation}
where the two squares are pullbacks.
Commutativity of~\eqref{depel:psi} implies $\partial^0_p.\partial^0_{\partial^0_p}.\psi.\PM_1 = f.\PM_0$,
therefore there exists the arrow
$\univ{\PM_0, \psi.\PM_1} \colon M \to \bsch{f}{\,\pob(\partial^0_p)}$ and we can define
$\phi \coloneqq v.i.\univ{\PM_0, \psi.\PM_1}$.

Commutativity of~\eqref{wspobtr:sq} and~\eqref{depel:v} imply
\[	\partial_M.v.i = \univ{\bsch{f}{(\partial^0_p.\partial^0_{\partial^0_p})},
				\partial_{\partial^0_p}.\bsch{(\partial^0_p.\partial^0_{\partial^0_p})}{f}}
	\qquad\text{and}\qquad
	v.i.\univ{\PM_0, \rar_{\partial^0_p}.\PM_1} = \rar_M		\]
Therefore the equations in~\eqref{depel:req} follow from commutativity of~\eqref{depel:psi} and~\eqref{depel:fill}
as follows:
\[\begin{split}
\univ{\bsch{f}{(\partial^0_p.\partial^0_{\partial^0_p})}, \partial_{\partial^0_p}.
	\bsch{(\partial^0_p.\partial^0_{\partial^0_p})}{f}}.\univ{\PM_0, \psi.\PM_1}
	& =		\univ{\PM_0, \partial_{\partial^0_p}.\psi.\PM_1}					\\
	& =		\univ{\PM_0, \univ{\rar_p.f.\PM_0, \PM_1}}						\\
	& =		\univ{\univ{\PM_0, \rar_p.f.\PM_0}, \univ{\PM_0, \PM_1}}			\\
	& =		\univ{\univ{\id_V, \rar_p.f}.\PM_0, \id_M}
\end{split}\]
and
\[\begin{split}
\univ{\PM_0, \psi.\PM_1} \univ{\id_V, \rar_p.f}
		& =		\univ{\id_V, \rar_{\partial^0_p}.\rar_p.f}					\\
		& =		\univ{\PM_0, \rar_{\partial^0_p}.\PM_1}.\univ{\id_V, \rar_p.f}
\end{split}\]
\end{proof}

We are now in a position to prove the existence of a weak factorization system.

\begin{theor}											\label{thm:wfs}
Every tribe with weakly stable path objects $(\catm, \Am)$
admits a weak factorization system $(\Lm, \Rm)$
defined by $\Lm \coloneqq {}^{\boxslash}\!\Am$ and $\Rm \coloneqq \Lm^{\boxslash}$.
\end{theor}
\begin{proof}
The factorization of an arrow $f \colon V \to Y$ is given by
\[	f = (\partial^1_Y.\PM_1).\univ{\id_V, \rar_Y.f},	\]
where $\univ{\id_V, \rar_Y f}$ is the arrow defined by the universal property of the
mapping path object of $f$:
\begin{equation}										\label{wfs:def1}
\xycenter{C=3em}{
	V 	\ar@/_1pc/[ddr]_{\id_X}		\ar@/^1pc/[drr]^{\rar_Y.f}	\ar[dr]	&&	\\
	&	\mpob(f) 	\ar[d]	\ar[r]	&	\pob(Y) 		\ar[d]^{\partial^0_Y}	\\
	&	V 		\ar[r]_f 			&	Y 	}
\end{equation}
\cref{lem:depel} implies $\univ{\id_V, \rar_Y.f} \in \Lm$.
In order to show that $\partial^1_Y.\PM_1$ is in \Rt observe that,
for the stability property of pullbacks, the following square is a pullback
\begin{equation}										\label{wfs:def2}
\xycenter{=3.5em}{
	\mpob(f)		\ar[d]^{\univ{\PM_0, \, \partial^1_Y.\PM_1}}		\ar[r]_{\PM_1}		&
						\pob(Y) 		\ar[d]^{\partial_Y}			\\
	V \times Y 		\ar[r]_{f \times \id_Y} 		&	Y \times Y 			}
\end{equation}
In particular, point~\ref{def:tribe:b} of \cref{def:tribe} implies
$\univ{\PM_0, \, \partial^1_Y.\PM_1} \in \Am$ and, in turn,
$\partial^1_Y.\PM_1 \in \Am \subset \Rm$.

Finally, the Weak Orthogonality Axiom follows immediately from \cref{lem:wort}
\end{proof}

We now give a characterization of the arrows in the right class \Rt
as we already did for the left class in \cref{prop:lftcl}.

\begin{prop}											\label{prop:rgtcl}
Let $(\catm, \Am)$ be a tribe with weakly stable path objects and
$f \colon V \to Y$ an arrow in \catt.
Then $f \in \Rm \coloneqq (^\boxslash\!\Am)^\boxslash$ if and only if
it has a transport arrow, that is, a diagonal filler for the following square.
\begin{equation}										\label{rgtcl:trsq}
\xycenter{=3em@C=3.5em}{
	V 	\ar[d]_{\univ{\id_V,\, \rar_Y.f}}	\ar@{=}[r]	&	V 	\ar[d]^f 	\\
	\mpob(f)			\ar[r]_(.6){\partial^1_Y.\PM_1}		&	Y 		}
\end{equation}
\end{prop}
\begin{proof}
If $f \in \Rm$, then the existence of a diagonal filler for~\eqref{rgtcl:trsq} is ensured by
\cref{lem:depel}, which implies $\univ{\id_V,\, \rar_Y.f} \boxslash f$.

Vice versa, consider the following lifting problem
\begin{equation}										\label{rgtcl:lp}
\xycenter{=3em}{
	W 	\ar[d]_g 	\ar[r]^h 	&	V 	\ar[d]^f 	\\
	Y 		\ar@{=}[r]	&	Y 			}
\end{equation}
where $g \in {}^\boxslash\!\Am$.
Because of the characterization given by \cref{prop:lftcl},
we know that there are $r \colon Y \to W$ and $\phi \colon Y \to \pob(Y)$ such that
\[		r.g = \id_W,		\qquad			\partial_Y.\phi = \univ{g.r, \id_Y}
		\qquad\text{and}\qquad			\phi.g = \rar_Y.g.			\]
This allows us to define the arrow $\univ{h.r, \phi} \colon Y \to \mpob(f)$.
A diagonal filler for~\eqref{rgtcl:lp} is then given by $j.\univ{h r, \phi} \colon Y \to V$,
where $j \colon \mpob(f) \to V$ is a transport arrow for $f$.
\end{proof}

Notice also that, if $f$ has a transport arrow $j \colon \mpob(f) \to V$,
then it has a transport arrow over every arrow $p \colon Y \to X$ in \At.
This is a consequence of the previous Proposition and of \cref{lem:depel},
which together imply $\univ{\id_V, \rar_p.f} \boxslash f$.

\section{The syntactic category of \mltt}							\label{sec:syntc}

We consider a version of \mltt \cite{MLo75,MLo84}, which we denote by \Tid,
equipped the usual structural rules and with the only identity type as type former,
whose rules are given in \cref{tab:idrules}.
For more details about the syntax of \mltt see~\cite{NPS90}.
We write $\bnd{x}B$ ($\bnd{x}b$) to denote the type $B$ (term $b$) with the variable $x$ bound and
$B\subst{a}{x}$ ($b\subst{a}{x}$) to denote the type (term) obtained by
substituting the term $a$ for all the free occurrences of the variable $x$ in $B$ (in $b$).

As usual, we assume that derivations given by all the rules we present can be made also relative
to chains of assumptions, called contexts.
Recall that a \emph{context} $\Phi$ of length $n \geq 0$ is a list of declarations of the form
\[	\Phi = (x_1 : A_1,\, x_2 : A_1,\dots ,\, x_n:A_n)	\]
such that all the variables $x_1, \dots, x_n$ are different and the following judgments are derivable.
\begin{align}
& A_1 \type, 										\nonumber \\
& (x_1 : A_1)\ A_1 \type,								\nonumber \\
& \qquad \vdots 										\label{eq:cxt} \\
& (x_1 : A_1, \dots ,\, x_{n-1} : A_{n-1})\ A_n \type. 					\nonumber
\end{align}
When $n = 0$ we have the empty context $()$.
We write $\Phi \cxt$ to mean that $\Phi$ is a context, that is,
as a shorthand for the judgments in \eqref{eq:cxt}.
If the first judgment $A_1 \type$ (and then all the others) depends on a context $\Gamma$,
we write $(\Gamma)\, \Phi \cxt$ and say that $\Phi$ is a \emph{dependent context}.
Notice that, given a dependent context $(\Gamma)\, \Phi \cxt$, we then obtain $(\Gamma, \Phi) \cxt$
by concatenating the judgments they stands for.

\begin{table}[ht]
\caption{Derivation rules for identity types}							\label{tab:idrules}
\centering
\begin{tabular}{c}
\toprule \\
\multicolumn{1}{l}{Formation} \\ \\
\begin{prooftree}
	A \type 	\qquad	 a : A \qquad b : A \justifies \Id{A}{a,b} \type
\end{prooftree}
\\ \\
\multicolumn{1}{l}{Introduction} \\ \\
\begin{prooftree}
	A \type 	\qquad	a : A \justifies \refl_A(a) : \Id{A}{a,a}
\end{prooftree}
\\ \\
\multicolumn{1}{l}{Elimination} \\ \\
\begin{prooftree}
	\[ (x : A,\, y : A,\, u : \Id{A}{x,y},\, z : \Theta)\ C \type
		\justifies \thickness=0em
	p : \Id{A}{a,b}	\qquad			%
		(x' : A,\, z' : \Theta\subst{x',x',\refl(x')}{x,y,u})\ d : C\subst{x',x',\refl(x'),z'}{x,y,u,z}	\]
	\justifies
	(w : \Theta\subst{a,b,p}{x,y,u})\ \J(a, b, p, \bnd{x',z'}d, z) : C\subst{a,b,p,w}{x,y,u,z}
\end{prooftree}
\\ \\
\multicolumn{1}{l}{Computation} \\ \\
\begin{prooftree}
	\[ (x : A,\, y : A,\, u : \Id{A}{x,y},\, z : \Theta)\ C \type
		\justifies \thickness=0em
	a : A 	\qquad
		(x' : A,\, z' : \Theta\subst{x',x',\refl(x')}{x,y,u})\ d : C\subst{x',x',\refl(x'),z'}{x,y,u,z}	\]
	\justifies
	(w : \Theta\subst{a,a,\refl(a)}{x,y,u})\
			\J(a, a, \refl(a), \bnd{x',z'}d, w) = d\subst{a,w}{x',z'} : C\subst{a,a,\refl(a),w}{x,y,u,z}
\end{prooftree}
\\ \\
\bottomrule
\end{tabular}
\end{table}

Since we are working in a theory with the only identity type, we assume an elimination rule,
called parametric, which is stronger than that one presented in~\cite{NPS90},
for it allows the presence of a parametric context $\Theta$ within the derivation.
This parametric elimination can be deduced from the usual one in presence of $\Pi$ types~\cite{Ja99}.

In addition to the rules in~\cref{tab:idrules}, there are also rules expressing the congruence of definitional
equality with respect to the constants Id, $\refl$ and $\J$, which we leave implicit.
Instead, we explicitly state in~\cref{tab:idsubst} the rules for substituting terms for variables in such constants.
As usual, stating these rules, we assume that no occurrence of free variables in $a$
becomes bound when substituting $a$ in another term, just renaming variables if needed.
To increase readability, we have omitted from the premises of the last rule of~\cref{tab:idsubst} the judgment
\[ (x : A, y_1 : B, y_2 : B, u : \Id{B}{y_1, y_2}, z : \Theta)\ C \type. \]
%

\begin{table}[ht]
\caption{Substitution rules for identity types}							\label{tab:idsubst}
\resizebox{\columnwidth}{!}{%
\begin{tabular}{c}
\toprule \\
\multicolumn{1}{l}{Id-substitution}\\ \\
\begin{prooftree}
	(x : A)\ B \type \qquad (x : A)\ b_1 : B \qquad (x : A)\ b_2 : B \qquad a : A
	\justifies
	\Id{B}{b_1, b_2}\subst{a}{x} = \Id{B\subst{a}{x}}{b_1\subst{a}{x}, b_2\subst{a}{x}} \type
\end{prooftree}
\\ \\
\multicolumn{1}{l}{\refl-substitution}\\ \\
\begin{prooftree}
	(x : A)\ B \type \qquad (x : A)\ b : B \qquad a : A
	\justifies
	\refl_{B}(b)\subst{a}{x} = \refl_{B\subst{a}{x}}(b\subst{a}{x}) :	%
								\Id{B\subst{a}{x}}{b\subst{a}{x}, b\subst{a}{x}}
\end{prooftree}
\\ \\
\multicolumn{1}{l}{\J-substitution}\\ \\
\begin{prooftree}
	\[\[ (x : A)\ B \type	\quad	(x : A)\ b_1 : B 	\quad 	(x : A)\ b_2 : B 	\quad	%
									(x : A)\ p : \Id{B}{b_1, b_2}
	\justifies \thickness=0em	
	a : A \qquad (x : A, y : B, z' : \Theta\subst{y,y,\refl(y)}{y_0,y_1,u})\	%
									d : C\subst{y,y,\refl(y),z'}{y_0,y_1,u,z} \]
	\justifies
	(w : \Theta\subst{b_1,b_2,p}{y_0,y_1,u}\subst{a}{x})\ %
								\J(b_1, b_2, p, \bnd{y,z'}d, z)\subst{a,w}{x,z} =
	\phantom{C(a, b_1\subst{a}{x}, b_2\subst{a}{x}, p\subst{a}{x}, w)} \]
	\justifies \thickness=0em					\phantom{(\Theta(b_1(a), b_2(a)))}
		= \J(b_1\subst{a}{x}, b_2\subst{a}{x}, p\subst{a}{x}, \bnd{y,z'}(d\subst{a}{x}), w) : %
				C\subst{b_1,b_2,p,w}{y_0,y_1,u,z}\subst{a}{x}
\end{prooftree}
\\
\phantom{a}\\
\bottomrule
\end{tabular} }
\end{table}

Using the structural rules of \mltt it is possible to define a category,
called syntactic category, whose construction we briefly recall.
For details see~\cite{Pi01}.

Given two contexts $\Phi$ and $\Psi = (y_1 : B_1, \dots, y_m : B_m)$, a \emph{context morphism}
from $\Phi$ to $\Psi$ in an m-tuple $(b_1, \dots, b_m)$ such that all the following judgments are derivable.
\begin{align*}
& (\Phi)\ b_1 : B_1,						\\
& (\Phi)\ b_2 : B_2\subst{b_1}{y_1}, 		\\
& \qquad \vdots \\
& (\Phi)\ b_m : B_m\subst{b_1, \dots, b_{m-1}}{y_1, \dots, y_{m-1}}.
\end{align*}
We write $(\Phi)\ (b_1, \dots, b_m) : \Psi$ to mean that $(b_1, \dots, b_m)$
is a context morphism from $\Phi$ to $\Psi$.

An important class of context morphisms, called \emph{dependent projections},
are those which drop some variables of a context, that is, context morphisms of the form
\[ (\Phi, \Psi)\ (x_1, \dots, x_n) : \Phi \]
where $\Psi$ is a context depending on $\Phi$, and $(x_1, \dots, x_n)$ are the variables in $\Phi$.

We can then define an equivalence relation on contexts and context morphisms, which extends definitional equality,
identifying two contexts (contexts morphisms) if they coincide up to variable renaming.

\begin{defin}
Objects and arrows of the \emph{syntactic category} \syntc of a type theory \T are equivalence classes of
contexts and context morphisms respectively.
Identity arrows are defined in the obvious way, whilst the composition of two arrows
$a \colon \Phi \to \Theta$ and $b \colon \Theta \to \Psi$ is defined by means of
substitution within representatives of equivalence classes:
if $(\Phi)\ (a_1, \dots, a_k) : \Theta$ and $(\Theta)\ (b_1, \dots, b_m) : \Psi$ are such representatives and
$z_1, \dots, z_k$ are the free variables in the context $\Theta$, the composite arrow is the equivalence class of
the context morphism given by the judgments
\[ (\Phi)\ b_i\subst{a_1, \dots, a_k}{z_0, \dots, z_k} : %
	B_i\subst{b_1, \dots, b_{i-1}}{y_1, \dots, y_{i-1}}\subst{a_1, \dots, a_k}{z_0, \dots, z_k}, \]
for $i = 1, \dots, m$, where $\Psi = (y_1 : B_1,\dots,\, y_m : B_m)$.
\end{defin}

In what follows, we will not use equivalence classes but just identify two contexts (context morphisms)
up to variable renaming.
Also, since we will only work with contexts, and not with types, we denote the
former with capital letters of Latin alphabet, instead of Greek one.

Gambino and Garner have proved in~\cite{GG08} that, exploiting the rules for the identity type,
it is possible to endow the syntactic category \syntcid associated to the type theory \Tid
with a weak factorization system, called identity type weak factorization system.

\begin{theor}[Gambino and Garner]
Let \Dt denotes the set of all the dependent projections in \syntcid.
The pair $(\Lm, \Rm)$, where $\Lm = {}^\boxslash\Dm$ and $\Rm = \Lm^\boxslash$,
is a weak factorization system on \syntcid.
\end{theor}

The Weak Orthogonality Axiom follows from \cref{lem:wort} whereas,
given two contexts $X$ and $Y$, the factorization of a context morphism
$(x : X)\ f : Y$ is
\begin{equation}										\label{eq:itwfs}
\xycenter{C=6em}{
	X	\ar[r]^(.3){(x,\, f,\, \refl(f))}	&	(x : X,\, y : Y,\, u : \Id{Y}{f, y})
								\ar[r]^(.7){(y)}	&	Y	}
\end{equation}
where $\Id{Y}{f,y}$ is the so-called identity context, which can be defined from the rules for identity types
as shown in~\cite{Ga09}.
The right-hand arrow is obviously a dependent projection, the hard part of the proof is in proving that the
left-hand one is in \Lt.

We now prove that the pair (\syntcid, \Dt) is a tribe with weakly stable path objects, and that the weak
factorization system given by \cref{thm:wfs} is the identity type weak factorization system.

\begin{prop}											\label{prop:sc2wspob}
The pair \textup{(\syntcid, \Dt)} is a tribe with weakly stable path objects.
In particular, the choice of path objects is given by identity contexts.
\end{prop}
\begin{proof}
First of all, the empty context $() \cxt$ is a terminal object in \syntcid,
terminal arrows are dependent projections dropping all the variables of a context,
whilst identity arrows are dependent projections which do not drop any variable.

Given a dependent projection $(x : X, y : E) \to (x : X)$ and a context morphism $t \colon X' \to X$,
the following square is a pullback
\begin{equation}										\label{sc2wspob:plbk}
\xycenter{=3em}{
	(x' : X',\, y' : E\subst{t}{x})	\ar[d]	\ar[r]^(.55){(t,\, y')}	&	(x : X,\, y : E)	\ar[d]	\\
	(x' : X')				\ar[r]^t		&	(x : X)				}
\end{equation}
where the two vertical arrows are dependent projections.
Indeed, if $f \colon Z \to X'$ and $(g_0, g_1) \colon Z \to (X, E)$ are context morphisms such that $g_0 = t.f$,
the universal arrow is the context morphism $(f, g_1)\colon Z \to (X,E\subst{t}{x})$.

Finally, the composite of two dependent projections is obviously a dependent projection.
Therefore, (\syntcid, \Dt) is a tribe.

We now show how to associate a path object to every dependent projection.
Let
\[ p \colon (x : X,\, y : E) \to (x : X), \]
be a dependent projection.
The diagonal $\Delta_p$  is the context morphism
\[ (x, y, y) \colon  (x : X, y : E) \to  (x : X, y_1 : E, y_2 : E). \]
Define
\[ \pob(p) := (x : X,\, y_1 : E,\, y_2 : E,\, u : \Id{E}{y_1, y_2}), \]
\[ \rar_p := (x, y, y, \refl(y)) \qquad\text{and}\qquad \partial_p := (x, y_1, y_2), \]
where $\Id{E}{y_1, y_2}$ is the (dependent) identity context of the context $E$,
for which one can derive rules analogous to those given in \cref{tab:idrules},
as proved in~\cite{Ga09}.
These definitions yield of course a factorization of the diagonal,
and $\partial_p \in \Dm$ follows from definition.

Let us prove that every left lifting problem for a pullback of $\rar_p$ over \Dt has a solution in \syntcid.
Let $(v : \pob(p),\, z : B) \to \pob(p)$ be a dependent projection
(where we use $v : \pob(p)$ as a shorthand for $(x,y_1,y_2,u) : \pob(p)$) and
\[\xycenter{R=3em@C=5em}{
	(x : X,\, y : E,\, z' : B\subst{x, y, y, \refl(y)}{v})	\ar[d]		\ar[r]^(.55)k		&
								(v : \pob(p),\, z : B)	\ar[d]		\\
	(x : X,\, y : E)	\ar[r]^{\qquad \rar_p}	& 	\pob(p) }				\]
a pullback, where $k = (\rar_p, z') = (x, y, y, \refl(y), z')$.

Let us then consider the following lifting problem
\begin{equation}										\label{sc2wspob:lp}
\xycenter{=3em}{
	(x : X,\, y : E,\, z' : B\subst{x, y, y, \refl(y)}{v})	\ar[d]_k	\ar[r]^(.55)d
						&	(v : \pob(p),\, z : B,\, w : C)	\ar[d]	\\
	(v : \pob(p),\, z : B) \ar@{=}[r]	&	(v : \pob(p),\, z : B) }
\end{equation}
where the right arrow is the dependent projection dropping the variable $w$.
Commutativity of~\eqref{sc2wspob:lp} implies
\[	d = (x, y, y, \refl(y), z', d_C),	\]
for some context morphism $d_C$ such that the judgment
\[	(x : X,\, y : E, \, z' : B\subst{x, y, y, \refl(y)}{v})\ d_C : C\subst{x, y, y, \refl(y),z'}{v,z}.	\]
holds.
Applying the elimination for identity contexts to it, we obtain
\[	\J(y_1, y_2, u, \bnd{y,z'}d_C, z) : C 		\]
for $(x, y_1, y_2, u) : \pob(p)$ and $z : B$, and the computation rule yields
\[ \J(y, y, \refl(y),  \bnd{y,z'}d_C, z') = d_C : C\subst{x, y, y, \refl(y),z'}{v,z}. \]
for $x : X,\, y : E,\, z' : B\subst{x, y, y, \refl(y)}{v}$.
Thus the context morphism
\[ 	(v : \pob(p),\, z : B)\ (x, y_1, y_2, u, z, \J(y_1, y_2, u, \bnd{y,z'}d_C, z))
									: (v : \pob(p),\, z : B,\, w : C)	\]
is a diagonal filler for the lifting problem~\eqref{sc2wspob:lp}: commutativity of the lower triangle is obvious
whereas commutativity of the upper one follows from the computation rule.
Therefore, the tribe (\syntcid, \Dt) has path objects.

To see that this choice is weakly stable in the sense of \cref{def:wspobtr},
let $f \colon X \to Y$ be a context morphism.
According to the choice for the pullback square in~\eqref{sc2wspob:plbk}, we have
\[	\bsch{f}{E} = (x : X,\, z' : E\subst{f}{y}),	\quad
			\bsch{f}{p} = (x),	\quad\text{and}\quad	\bsch{p}{f} = (f, z'),	\]
while the pullback object of $\pob(p)$ along $f$ is
\[\begin{split}
\bsch{f}{\,\pob(p)}	&	=			%
		 (x : X,\, z'_1 : E\subst{f}{y},\, (z_2 : E,\, u' : \Id{E}{z_1, z_2})\subst{f,z'_1}{y,z_1})	\\
	&	=	(x : X,\, z'_1 : E\subst{f}{y},\, z'_2 : E\subst{f}{y},\, %
							u' : (\Id{E}{z_1, z_2})\subst{f,z'_1,z'_2}{y,z_1,z_2})
\end{split}\]
and the two associated base change arrows are
\[	\bsch{f}{(p.\partial_p^0)} = (x)		\qquad\text{and}\qquad
					\bsch{(p.\partial_p^0)}{f} = (f, z'_1, z'_2, u').	\]
Therefore, from the construction in the syntactic category of the universal arrow of a pullback, we have
\[\begin{split}
\univ{\bsch{f}{p}, \, \rar_p.\bsch{p}{f}}	& =	\univ{(x), \, (f,z',z',(\refl_E(z))\subst{f,z'}{y,z})}	\\
						& =	(x, z', z', (\refl_{E}(z))\subst{f,z'}{y,z})
\end{split}\]
and
\[\begin{split}
\univ{\bsch{f}{(p.\partial_p^0)}, \, \partial_p.\bsch{(p.\partial_p^0)}{f}}	%
									& =	\univ{(x), \, (f, z'_1, z'_2)}	\\
									& =	(x, z'_1, z'_2)
\end{split}\]
On the other hand, the factorization of $\Delta_{\bsch{f}{p}}$ is given by
\[	\pob(\bsch{f}{p}) = (x : X,\, z'_1 : E\subst{f}{y},\, z'_2 : E\subst{f}{y},\,	%
									 u' : \Id{E\subst{f}{y}}{z'_1, z'_2})	\]
\[	\rar_{\bsch{f}{p}} = (x, z', z', \refl_{E\subst{f}{y}}(z'))		\qquad\text{and}\qquad	%
									\partial_{\bsch{f}{p}} = (x, z'_1, z'_2) \]
and the substitution rules for identity types ensure
\[	(\Id{E}{z_1, z_2})\subst{f,z'_1,z'_2}{y,z_1,z_2}	=	\Id{E\subst{f}{y}}{z'_1, z'_2}	\]
and
\[	(\refl_{E}(z))\subst{f,z'}{y,z} = \refl_{E\subst{f}{y}}(z').		\]
Therefore $\bsch{f}{\,\pob(p)} = \pob(\bsch{f}{p})$ and it suffices to take the identity context morphism
as the arrow $i$ in order to make diagram~\eqref{wspobtr:sq} commute.
\end{proof}

The third substitution rule for identity types, which we did not use here, is needed in the case a choice of
diagonal filler is assumed in the definition of tribe with path objects and such choice required to be stable
under pullback in the definition of tribe with weakly stable path objects.
We did so in~\cite{Em13}, also assuming a stronger notion of stability under pullback of the structure
given by path objects.
This assumption amounts, in the current formulation, to requiring the arrow $i$ of \cref{def:wspobtr}
to be in isomorphism.

Let us conclude this section by showing that the weak factorization system given by \cref{thm:wfs}
is precisely the identity type weak factorization system.
The two classes of arrows \Lt and \Rt are obviously the same.
Consider then a context morphism $f \colon X \to Y$.
The factorization in \cref{thm:wfs} is obtained through the mapping path object of $f$
over the terminal arrow $Y \to ()$:
\[\xymatrix@C=4em@R=3em{
	X	\ar@/_1.5pc/[ddr]_{\id_X}	\ar[dr]^{\univ{\id_X,\, \rar_Y.f}} \ar@/^1.5pc/[drr]^{\rar_Y.f}	&&\\
	&	\mpob(f)	\ar[d]^{\PM_0}	\ar[r]^{\PM_1}		&	\pob(Y)	\ar[d]^{\partial_Y}	\\
	& 	X			\ar[r]^f			&	Y				}\]
where the mapping path object is the context
\[	\mpob(f) = (x : X, y : Y, u : \Id{Y}{f, y})	\]
because of the choice of pullback and the first substitution rule for identity types.

The arrow $\partial^1_Y \PM_1$ is the dependent projection
\[	(y)	\colon	(x : X, y : Y, u : \Id{Y}{f, y})	\to	Y	\]
and the universal arrow $\univ{\id_X, \rar_Y.f} \colon X \to \mpob_p(f)$ is
\[	\univ{(x), (f, f, \refl_Y(f))}	=	(x, f, \refl_Y(f)).	\]
We have then obtained exactly the factorization of $f$ described in~\eqref{eq:itwfs}.
Thus, the weak factorization system on (\syntcid, \Dt) given by \cref{thm:wfs}
is the identity type weak factorization system.

\section{Path object categories}								\label{sec:pobc}

A path object category is a notion introduced by Benno van den Berg and Richard Garner in~\cite{vdBG11}.
In that paper they first prove that every path object category is a model of \mltt with identity types,
and prove then that the categories of groupoids, of chain complexes over a ring,
of topological spaces and of simplicial sets are path objects categories.
As they say, such structure has been defined 
in order to have an abstract characterization of a number of common features shared by those categories,
features which are instrumental in showing that they provide models of \mltt with identity types.

\begin{defin}											\label{def:pobc}
A \emph{path object category} is a finitely complete category \catt equipped with the following structure:
\begin{enumerate}[label=(\emph{\alph*})]
\item An endofunctor $\Mm \colon \catm \to \catm$ which preserves pullbacks, and natural transformations			\label{pobc:a}
	\[	\rar \colon \Idf \natto	\Mm,	\qquad	\src,\trg \colon \Mm \natto \Idf,	\qquad
		\cmp \colon \univ{\Mm, \Mm} \natto \Mm	\qquad\text{and}\qquad	\tau \colon \Mm \natto \Mm	\]
	such that, for every object $X$ of \catt, the diagram
	\[\xycenter{=4em}{
		X 	\ar[r]|-{\rar_X}	&	\Mm X 	\ar@<-1ex>[l]_{\src_X}	\ar@<1ex>[l]^{\trg_X}	&
			\Mm X \presc{\src_X}{\times_{\trg_X}} \Mm X		\ar[l]_(.6){\cmp_X}		}\]
	is an internal category, and the arrow $\tau_X \colon \Mm X \to \Mm X$ is an involution
	which defines an internal identity-on-objects isomorphism between $\Mm X$ and its opposite category.
\item A strength for the endofunctor \Mt~\cite{Ko70}						\label{pobc:b}
	\[	\alpha_{X,Y} \colon \Mm X \times Y \to \Mm(X \times Y)	\]
	with respect to which $\rar, \src, \trg, \cmp$ and $\tau$ are strong natural transformation.
\item A strong natural transformation $\eta \colon \Mm \natto \Mm\Mm$ such that			\label{pobc:c}
	\begin{align*}
	\src_{\Mm X}.\eta_X 	& = \id_X,					\\
	\trg_{\Mm X}.\eta_X 	& = \rar_X.\trg_X,				\\
	\Mm\src_X.\eta_X 	& = \id_X,					\\
	\Mm\trg_X.\eta_X 	& = \alpha_{\term,X}.\univ{\Mm !, \trg_X},	\\
	\eta_X.\rar_X 		& = \rar_{\Mm X}.\rar_X.
	\end{align*}
\end{enumerate}
\end{defin}

The motivating example that van den Berg and Garner have in mind is the category of topological spaces 
where the endofunctor \Mt is intended to provide a choice of path space for every space $X$.
However, the obvious choice $\Mm X \coloneqq X^{[0,1]}$ fails to satisfy the category axioms they require in
point~\ref{pobc:a}.
For this reason, they take instead $\Mm X$ to be the \emph{Moore path space},
i.e.\ the set of paths with arbitrary lengths
\[	\Set{ (l, \phi) \in \preal \times X^{\preal}	|	(\forall t \geq l) \phi(t) = \phi(l) },	\]
where $\preal$ denotes the set of non-negative real numbers, equipped with the topology induced
by the euclidean topology on $\preal$ and the compact-open topology on $X^{\preal}$.
The category structure is then given as follows.
The identity $\rar_X(x)$ is the constant path at $x \in X$ of length $0$.
Source $\src_X$ and target $\trg_X$ of a path $(l, \phi)$ are $\phi(0)$ and $\phi(l)$, respectively.
The composition $\cmp_X$ of two paths $(l,\phi)$ and $(m, \psi)$ such that $\phi(l) = \psi(0)$
is the path $\chi$ of length $l + m$ defined by
\[	\chi(t) \coloneqq
\begin{cases}
\phi(t),	&	\text{if $t \leq l$,}	\\
\psi(t - l),	&	\text{if $t \geq l$.}
\end{cases}\]
Finally, the involution $\tau_X$ is the operation which revert paths,
that is, takes a path $(l, \phi)$ to the path $\phi^o$ of same length defined by
\[	\phi^o(t) \coloneqq
\begin{cases}
\phi(l - t),	&	\text{if $t \leq l$,}	\\
\phi(0),	&	\text{if $t \geq l$.}
\end{cases}\]

The notion of strength $\alpha_{X,Y} \colon \Mm X \times Y \to \Mm(X \times Y)$ captures the idea
of a constant path of arbitrary length.
Indeed, such a strength is determined, up to natural isomorphism, by the components of the form $\alpha_{\term, X}$.
Now, in the topological example, $\Mm\term$ is homeomorphic to $\preal$,
therefore we can think of $\alpha_{\term,X}$ as taking a length $l \in \preal$ and a point $x \in X$
to the path of length $l$ constant at $x$.
The component $\alpha_{X,Y}$ then takes a pair $((l, \phi), y)$ to the path $(l, \psi)$ in $X \times Y$
defined by $\psi(t) \coloneqq (\phi(t), y)$.

This strength is needed to encode the last piece of structure, that is, a way to contract a path onto its endpoint.
This is done by requiring the existence of the natural transformation $\eta$ in point~\ref{pobc:c}.
In the topological example, such a contraction is given by a path of paths $\eta_X(\phi)$ of the following form
\[\xycenter{=2.5em@C=3.5em}{
	x_0	\ar[d]_\phi	\ar[r]^{\Mm \src_X.\eta_X(\phi)}	&	x_1	\ar[d]^{\rar_X(x_1)} \\
	x_1	\ar[r]_{\Mm \trg_X.\eta_X(\phi)}	\ar@{}[ur]|-{\displaystyle{\overset{\eta_X(\phi)}{\Longrightarrow}}}	&	x_1	}\]
with $\src_{\Mm X}.\eta_X(\phi) = \phi$ and $\trg_{\Mm X}.\eta_X(\phi) = \rar_X(x_1)$.
To make such diagram commute,
one might try to take $\Mm \src_X.\eta_X(\phi) = \phi$ and $\Mm \trg_X.\eta_X(\phi) = \rar_X(x_1)$.
The problem arises from the fact that $\Mm\src_X$ and $\Mm\trg_X$ preserve path lengths,
but the length $l$ of $\phi$ is in general not equal to $0$.
The solution is then given, in the topological case, by the path of length $l$ constant at $x_1$ and,
in the abstract case, by the strength $\alpha_{\term,X}$ (see the fourth equation in point~\ref{pobc:c}).

We now show that a path object category is, in particular, a tribe with weakly stable path objects,
thus providing the latter with a number of example.

The argument is essentially the same as that one in Section 6.2 of \cite{vdBG11}.

\begin{theor}										\label{thm:pobc2wspob}
A path object category is a tribe with weakly stable path objects,
whose class \At is formed by all the arrows which have a transport arrow.
\end{theor}
\begin{proof}
Let \catt be a path object category.
In order to show that it is a tribe we first observe that, since it is finitely complete,
it has a terminal object and arbitrary pullbacks,
and define \At to be the class of arrows which have a transport arrow.
It is easy to see that a transport arrow for $\bsch{f}{p} \colon \bsch{f}{E} \to X$ is
$\univ{\trg_X.\PM_1, j \univ{\bsch{p}{f}.\PM_0, \Mm f.\PM_1}} \colon \mpob(\bsch{f}{p}) \to \bsch{f}{E}$,
where $j$ is a transport arrow for $p$,
whereas transport arrows for a terminal $X \to \term$ and for an iso $f \colon X \to Y$ are
$\PM_0 \colon X \times \Mm\term \to X$ and $\trg_X.\Mm f^{-1}.\PM_1 \colon \mpob(f) \to X$, respectively.

Let $p \colon E \to Y$ be an arrow in \At and define $\pob(p)$ as the pullback object in
\begin{equation}									\label{pobc2wspob:pob}
\xycenter{=3em@C=4em}{
	E	\ar@/_/[ddr]_{\univ{\rar_\term.!,\, p}}	\ar@{..>}[dr]	\ar@/^/[rrd]^{\rar_E}& &	\\
	&	\pob(p)		\ar[d]^{u_p}	\ar[r]_{v_p}		&	\Mm E	\ar[d]^{\Mm p}	\\
	&	\Mm\term \times Y	\ar[r]_{\alpha_{\term,Y}}	&	\Mm Y			}
\end{equation}
$\rar_p$ as the universal arrow in \eqref{pobc2wspob:pob} and
$\partial_p \coloneqq \univ{\src_E.v_p, \trg_E.v_p} \colon \pob(p) \to E \times_p E$.
This yields a factorization of the diagonal $\Delta_p$,
and the proof that $\partial_p \in \Am$ and $\rar_p \boxslash \Am$
is the content of Proposition 6.2.2 of \cite{vdBG11}.

The fact that every base change of  $\rar_p$ along an arrow in \At is in ${}^\boxslash\!\Am$ follows from
the fact that this choice of factorization is Frobenius,
i.e.\ ${}^\boxslash\!\Am$ is closed under base change along arrows in \At,
which is proved in Proposition 6.3.1.

Finally, the weak stability follows from Proposition 6.2.5, where van den Berg and Garner prove that
there is a natural isomorphism between $\bsch{f}{\,\pob(p)}$ and $\pob(\bsch{f}{p})$
which is compatible with $\rar_{\bsch{f}{p}}$ and $\partial_{\bsch{f}{p}}$ as required.
\end{proof}

\bibliographystyle{plain}
\bibliography{biblio-cat-id-ty}

\end{document}